\documentclass[reqno]{amsart} 
\usepackage{enumerate,amsmath,amssymb} 
\usepackage{color}
\usepackage{hyperref}
\theoremstyle{plain} 
 
\newtheorem{thm}{Theorem}[section] 
\newtheorem{theorem}[thm]{Theorem} 
 
\newtheorem{corollary}[thm]{Corollary} 
 
\newtheorem{lemma}[thm]{Lemma} 
\newtheorem{prop}[thm]{Proposition} 
\newtheorem{proposition}[thm]{Proposition} 
 
\theoremstyle{remark} 
\newtheorem{example}[thm]{Example} 
\newtheorem{remark}[thm]{Remark}

\theoremstyle{definition}

\def\al{{\alpha}}

\def\om{{\omega}}

\def\la{{\lambda}}
\let\La\Lambda

\def\si{{\sigma}}

\def\ga{{\gamma}}
\def\epsilon{{\varepsilon}}
\def\ep{{\varepsilon}}

\def\th{{\theta}}

\def\phi{{\varphi}}

\DeclareMathAlphabet{\doba}{U}{msb}{m}{n} 

\gdef\mH{\doba{H}}

\gdef\mR{\doba{R}}
\gdef\mS{\doba{S}}

\gdef\mZ{\doba{Z}}
\gdef\mM{\doba{M}}

\def\cF{\mathcal{F}}
\def\cS{\mathcal{S}}

\def\SO{{\mathop{\rm SO}}}
\def\SU{{\mathop{\rm SU}}}
\def\vol{{\mathop{\rm vol}}}
\def\Ric{{\mathop{\rm Ric}}}
\def\Isom{{\mathop{\rm Isom}}}
\def\Stab{{\mathop{\rm Stab}}}
\def\Scal{{s}}
\let\scal\Scal
\def\Id{{\mathop{\rm Id}}}

\let\<\langle 
\let\>\rangle 
\let\ti\tilde

\let\wihat\widehat

\newcommand{\definedas}{\mathrel{\raise.095ex\hbox{\rm :}\mkern-5.2mu=}}

\newcounter{mnotecount}[section]

\def \bi{b}
\def \ci{c}
\def\sh#1{\mathop{\rm sh}\nolimits_{#1}}

\begin{document}
%%%%%%%%%%%%%%%%%%%%%%%%%%%%%%%%%%%%%%%%%%%%%%%%%%%%%%%%%%%%%%%%%%%%%%%%%

%%%%%%%%%%%%%%%%%%%%%%%%%%%%%%%%%%%%%%%%%%%%%%%%%%%%%%%%%%%%%%%%%%%%%%%%%
%\begin{center}
%\framebox{\framebox{
%\vbox{This is project {\red \Project}\\
%Current version {\blue\Version}, from
%{\blue\Datum}, most recent changes by {\blue\Person}.}
%}}
%\end{center}
%%%%%%%%%%%%%%%%%%%%%%%%%%%%%%%%%%%%%%%%%%%%%%%%%%%%%%%%%%%%%%%%%%%%%%%%%

\title{Low-dimensional surgery and the Yamabe invariant} 
 
\author{Bernd Ammann} 
\address{Fakult\"at f\"ur Mathematik \\ 
Universit\"at Regensburg \\
93040 Regensburg \\ 
Germany}
\email{bernd.ammann@mathematik.uni-regensburg.de}
%\email{bernd.ammann@mathematik.uni-regensburg.de}

\author{Mattias Dahl} 
\address{Institutionen f\"or Matematik \\
Kungliga Tekniska H\"ogskolan \\
100 44 Stockholm \\
Sweden}
\email{dahl@math.kth.se}
%\email{dahl@math.kth.se}

\author{Emmanuel Humbert} 
\address{Laboratoire de Math\'ematiques et Physique Th\'eorique \\ 
Universit\'e de Tours \\
Parc de Grandmont \\
37200 Tours - France \\}
\email{Emmanuel.Humbert@lmpt.univ-tours.fr}
%\email{Emmanuel.Humbert@lmpt.univ-tours.fr}

\begin{abstract}
Assume that $M$ is a compact $n$-dimensional manifold and that $N$ is 
obtained by surgery along a $k$-dimensional sphere, $k\leq n-3$. The 
smooth Yamabe invariants $\sigma(M)$ and $\sigma(N)$ satisfy 
$\sigma(N)\geq {\rm min} (\sigma(M),\Lambda)$ for a constant 
$\Lambda>0$ depending only on $n$ and $k$. We derive explicit lower 
bounds for $\Lambda$ in dimensions where previous methods failed, namely 
for $(n,k) \in \{(4,1),(5,1),(5,2), (6,3), (9,1),(10,1)\}$. With methods 
from surgery theory and bordism theory several gap phenomena for smooth 
Yamabe invariants can be deduced.
\end{abstract}

% 
% 35J60 Nonlinear PDE of elliptic type
% 35P30 Nonlinear eigenvalue problems, nonlinear spectral theory for PDO
% 57R65 Surgery and handlebodies
% 58J50 Spectral problems; spectral geometry; scattering theory
% 58C40 Spectral theory; eigenvalue problems
%
% Old codim2 {53C27 (Primary) 55N22, 57R65 (Secondary)}
% Old mu2-MSC%%%% 53A30, 35J60(Primary) 35P30, 58J50, 58C40 (Secondary) 


\date{\today}

\maketitle
\tableofcontents

%%%%%%%%%%%%%%%%%%%%%%%%%%%%%%%%%%%%%%%%%%%%%%%%%%%%%%%%%%%%%%%%%%%%%%%%%
\section{Introduction and Results}
%%%%%%%%%%%%%%%%%%%%%%%%%%%%%%%%%%%%%%%%%%%%%%%%%%%%%%%%%%%%%%%%%%%%%%%%%

Let $(M,g)$ be a Riemannian manifold of dimension $n \geq 3$. Its 
scalar curvature will be denoted by $\Scal^g$. We define the Yamabe 
functional by 
\[
\cF^g (u) 
\definedas 
\frac{
\int_M \left( a_n |du|_g^2 + \Scal^g u^2 \right) \, dv^g}
{\left( \int_M |u|^{p_n} \, dv^g \right)^{\frac{2}{p_n}}} ,
\]
where $u \in C^\infty_c(M)$ does not vanish identically, and where 
$a_n \definedas \frac{4(n-1)}{n-2}$ and $p_n \definedas \frac{2n}{n-2}$. 
The {\it conformal Yamabe constant} $\mu(M,g)$ of $(M,g)$ is then defined 
by
\[
\mu(M,g) \definedas
\inf_{u \in C_c^{\infty}(M), u \not\equiv 0} \cF^g(u).
\]
This functional played a crucial role in the solution of the 
Yamabe problem which consists in finding a metric of constant scalar 
curvature in a given conformal class. For a compact manifold $M$ 
the {\em Yamabe invariant}
is defined by 
\[
\sigma (M) \definedas \sup \mu(M,g),
\]
where the supremum runs over all the metrics on $M$, or equivalently over 
all conformal classes on $M$. In order to stress that the Yamabe invariant 
only depends 
on the differentiable 
structure of $M$, it is often called the ``smooth Yamabe invariant of $M$''. 
One motivation for studying such an invariant is
given by the following well-known result 
\begin{prop} 
A compact differentiable manifold of dimension $n \geq 3$ admits a metric 
with positive scalar curvature if and only if $\si(M)>0$. 
\end{prop}
Note that all manifolds in this article are manifolds \emph{without boundary}.

We recall that classification of all compact manifolds of dimension $n \geq 3$ 
admitting a positive scalar curvature metric is a challenging open problem 
solved only in dimension $3$ by using Hamilton's Ricci flow and 
Perelman's methods. This is one reason why 
much work has been devoted to the study of $\si(M)$. 

One of the first goals should be to compute $\si(M)$ explicitly for 
some standard manifolds~$M$. This is unfortunately a problem out of range 
even for what could be considered the simplest examples. For example, the 
value of the Yamabe invariant is not known for quotients of spheres except 
for $\mR P^3$ (and the spheres themselves), for products of spheres of 
dimension at least $2$ and for hyperbolic spaces of dimension at least $4$.

One also could ask for general bounds for $\sigma(M)$. The fundamental 
one is due to Aubin,
\[
\sigma(M) \leq \sigma(S^n) = \mu(\mS^n) = n(n-1) \om_n^{2/n}.
\]
Here $\mS^n$ is the standard sphere in $\mR^{n+1}$, and its volume is 
denoted by $\om_n$.

Unfortunately, in dimension $n \geq 5$, not much more is known. Even the basic 
question whether there exists a compact manifold $M$ of dimension $n \geq 5$ 
satisfying $\sigma(M) \neq 0$ and $\sigma(M)\neq\sigma(S^n)$ is still open.
It would also be interesting to see whether the set
\[
\cS_n(0) \definedas 
\{\si(M) \mid \mbox{$M$ is a compact connected manifold of dimension $n$}\}
\] 
is finite or countably infinite, and whether $\cS_n(0)$ is dense in 
$(-\infty,\si(S^n)]$. Much more is known now about 
\[
\cS_n(i) \definedas 
\{\si(M) \mid \mbox{$M$ is a compact $i$-connected manifold of dimension $n$}\}
\] 
for $i\geq 1$, as we will see below.  

A useful tool for understanding the Yamabe invariant is to study its 
change under surgery type modifications of the manifold. The main results 
obtained this way are the following.
\begin{itemize} 
\item 
In 1979, Gromov-Lawson and Schoen-Yau independently proved that the
positivity of $\sigma(M)$ is preserved under surgery of dimension
$k \leq n-3$. One important corollary is that any compact simply connected
non-spin manifold of dimension $n \geq 5$ admits a positive scalar
curvature metric. Together with results by Stephan Stolz (1992) this implies  
$\cS_n(1)\subset (0,\si(S^n)]$ for $n\equiv 3,5,6,7$ modulo $8$, $n\geq 5$.
\item 
In 1987, Kobayashi proved that $0$-dimensional surgeries do not 
decrease $\sigma(M)$.
\item 
In 2000, Petean and Yun proved that if $N$ is obtained by a 
$k$-dimensional surgery ($k \leq n-3$) from $M$ then 
$\sigma(N) \geq \min(0, \sigma(M))$. This implies in particular
that if $M$ is simply connected and has dimension $n \geq 5$ then
$\sigma(M) \geq 0$. In other words $\cS_n(1)\subset [0,\si(S^n)]$ 
for all $n\geq 5$.
\end{itemize}
In \cite{ammann.dahl.humbert:13b} we proved a generalization of these 
three results.
\begin{theorem}[\cite{ammann.dahl.humbert:13b}, Corollary~1.4] 
\label{thm_surg}
If $N$ is obtained from a compact $n$-dimensional manifold~$M$ by 
a $k$-dimensional surgery, $k \leq n-3$, 
then 
\[
\sigma(N) \geq \min(\Lambda_{n,k}, \sigma(M))
\]
where $\Lambda_{n,k} > 0 $ depends only on $n$ and $k$. In addition, 
$\Lambda_{n,0} = \sigma(S^n)$.
\end{theorem}

As a corollary we see that $0$ is not an accumulation point of $\cS_n(1)$, 
$n\geq 5$, in other words we find that for any simply connected compact 
manifold~$M$ of dimension $n \geq 5$ 
\begin{itemize}
\item 
$\sigma(M) = 0$ if $M$ is spin and if its index in $KO_n$ does not vanish,
\item 
$\sigma(M) \geq \al_n$, otherwise, where $\al_n >0$ depends only on $n$. 
\end{itemize}
Many other consequences can be deduced, see 
\cite[Section~1.4]{ammann.dahl.humbert:13b}, but one could find 
these results unsatisfactory, since the constant $\Lambda_{n,k}$ were not 
computed in \cite{ammann.dahl.humbert:13b} unless for $k=0$. This effect 
was then reflected in the applications. For example, no explicit positive 
lower bound for the constant $\al_n$ above was known. The results 
in~\cite{ammann.dahl.humbert:13} and~\cite{ammann.dahl.humbert:p11b} 
yield explicit positive lower bounds for $\Lambda_{n,k}$ in the cases 
$2\leq k\leq n-4$. In order to apply standard surgery techniques, it would 
be helpful to have lower bounds in the cases $k=1$ and $k=n-3$.

The method established in the present article yields explicit positive 
lower bounds for all cases $k=1\leq n-4$ and in the cases $(n,k)=(6,3)$, 
$(n,k)=(5,2)$ and $(n,k)=(4,1)$. However it requires as input data a 
lower bound on the conformal Yamabe constant $\mu(\mR^{k+1}\times \mS^{n-k-1})$. 
Such input data is provided in \cite{petean.ruiz:11} and 
\cite{petean.ruiz:13} in the cases 
$(n,k) \in \{(4,1),(5,1),(5,2),(9,1),(10,1)\}$. Unfortunately their 
method has to be strongly modified for each pair of dimensions, and as 
a courtesy to us, Petean and Ruiz provided the above cases, as these are 
the ones which will lead to interesting applications in 
Section~\ref{section_top_appl}.

We obtain in 
%Section~\ref{section_top_appl}, 
Corollary~\ref{cor.dim.5} that $\cS_5(1)\subset (45.1,\si(S^5)]$, in other 
words: any compact simply connected manifold of dimension $5$ satisfies 
\[
45.1 < \si(M) \leq \mu(\mS^5) < 79.
\]
In the same way, Corollary~\ref{cor.dim.6} says that 
$\cS_6(1) \subset (49.9,\si(S^6)]$.

In dimensions $n\geq 7$ an unsolved problem persists for surgeries of 
codimension $3$, i.e.\ for $n=k-3$, 
see \cite{ammann.dahl.humbert:p11b} for details about this problem.

This problem can be avoided by restricting to $2$-connected manifolds. 
Together with results from  \cite{ammann.dahl.humbert:p11b} we obtain 
an explicit positive number $t_n$ such that any compact $n$-dimensional 
$2$-connected manifold $M$ with vanishing index, $n\neq 4$, satisfies 
$\si(M)\geq t_n$, see Table~\ref{fig.tn} and 
Proposition~\ref{prop.2c.lowbound}. We thus see 
$\cS_n(2) \subset \{0\} \cup [t_n,\si(S^n)]$ for all $n\neq 4$.

%%%%%%%%%%%%%%%%%%%%%%%%%%%%%%%%%%%%%%%%%%%%%%%%%%%%%%%%%%%%%%%%%%%%%%%%%
\subsection*{Acknowledgments}
%%%%%%%%%%%%%%%%%%%%%%%%%%%%%%%%%%%%%%%%%%%%%%%%%%%%%%%%%%%%%%%%%%%%%%%%%

We thank Jimmy Petean, Miguel Ruiz, and Tobias Weth for helpful comments. 
Much work on this article was done during a visit of Bernd Ammann and 
Mattias Dahl to the Max Planck Institute for Gravitational Physics 
(Albert Einstein Institute), Golm. We thank the institute for its 
hospitality. Emmanuel Humbert was partially supported by ANR-10-BLAN 0105.

%%%%%%%%%%%%%%%%%%%%%%%%%%%%%%%%%%%%%%%%%%%%%%%%%%%%%%%%%%%%%%%%%%%%%%%%%
\section{Preliminaries}
%%%%%%%%%%%%%%%%%%%%%%%%%%%%%%%%%%%%%%%%%%%%%%%%%%%%%%%%%%%%%%%%%%%%%%%%%

%%%%%%%%%%%%%%%%%%%%%%%%%%%%%%%%%%%%%%%%%%%%%%%%%%%%%%%%%%%%%%%%%%%%%%%%%
\subsection{Notation and model spaces} 
\label{sec.notation}
%%%%%%%%%%%%%%%%%%%%%%%%%%%%%%%%%%%%%%%%%%%%%%%%%%%%%%%%%%%%%%%%%%%%%%%%%
 
We denote the standard flat metric on $\mR^v$ by $\xi^v$. On the sphere 
$S^w \subset \mR^{w+1}$ the standard round metric is denoted by $\rho^w$.
The volume of $(S^w,\rho^w)$ is 
\[
\omega_w 
= \frac{2 \pi^{(w+1)/2}} {\Gamma\left(\frac{w+1}2\right)}.
\]

Let $\mH_\ci^v$ be the $v$-dimensional complete 1-connected Riemannian 
manifold with sectional curvature $-\ci^2$. The Riemannian metric on 
$\mH_\ci^v$ is denoted by $\eta_\ci^v$. We fix a point~$x_0$ in $
\mH_\ci^v$.

Next, we define the model spaces $\mM_\ci$ through 
$\mM_\ci \definedas \mH_\ci^v \times \mS^w$, which has the Riemannian 
metric $G_\ci \definedas \eta_\ci^v + \rho^w$. Note that in our previous 
articles~\cite{ammann.dahl.humbert:13b,ammann.dahl.humbert:13} 
we used the notation $\mM_\ci^{v+w, v-1}$ for $\mM_\ci$. 
Set $n \definedas v+w$.

%The main difficulty we encounter when trying to make the surgery 
%result in Theorem~\ref{thm_surg} more explicit is the complexity of the 
%definition of $\Lambda_{n,k}$. We will recall this definition now. 

Let $(N,h)$ be a Riemannian manifold of dimension $n$. Let $\Delta^h$ 
denote the non-negative Laplacian on $(N,h)$. For $i=1,2$ we let 
$\Omega^{(i)}(N,h)$ be the set of non-negative $C^2$ functions $u$
solving the Yamabe equation 
\begin{equation} \label{eq.conf} 
a_n \Delta^h u + \Scal^h u = \mu u^{p_n-1}
\end{equation}
for some $\mu = \mu(u) \in \mR$ and satisfying
\begin{itemize}
\item $u \not \equiv 0$,
\item $\|u\|_{L^{p_n}(N)} \leq 1$,
\item $u \in L^{\infty}(N)$,
\end{itemize}
and 
\begin{itemize}
\item $u \in L^2(N)$, for $i=1$, 
\end{itemize}
or
\begin{itemize}
\item $\mu(u) \|u\|^{p_n-2}_{L^{\infty}(N)} \geq
\frac{(n-k-2)^2(n-1)}{8(n-2)}$, for $i=2$. 
\end{itemize}
For $i=1,2$ we set
\[
\mu^{(i)} (N,h) \definedas \inf_{u \in \Omega^{(i)}(N,h)} \mu(u).
\]
In particular, if $\Omega^{(i)}(N,h)$ is empty then
$\mu^{(i)}(N,h)=\infty$.

Finally, the constants in the surgery theorem are defined as follows.
For integers $n \geq 3$ and $0 \leq k \leq n-3$ set 
\[
\La^{(i)}_{n,k}
\definedas 
\inf_{c \in [0,1]} 
\mu^{(i)} (\mM_\ci ) 
\]
and
\[
\La_{n,k} 
\definedas
\min \left\{ \La^{(1)}_{n,k},\La^{(2)}_{n,k}\right\}.
\]
where $v = k+1$ and $w = n-k-1$.

%%%%%%%%%%%%%%%%%%%%%%%%%%%%%%%%%%%%%%%%%%%%%%%%%%%%%%%%%%%%%%%%%%%%%%%%%
\subsection{Strategy of proof} 
\label{sec.strategy}
%%%%%%%%%%%%%%%%%%%%%%%%%%%%%%%%%%%%%%%%%%%%%%%%%%%%%%%%%%%%%%%%%%%%%%%%%

The strategy we have used to find lower bounds of $\Lambda_{n,k}$ is 
the following.
\begin{itemize}
\item 
First prove that $\La^{(2)}_{n,k} \geq \La^{(1)}_{n,k}$. This was the main 
result in \cite{ammann.dahl.humbert:p11b} which holds in the cases 
$k \leq n-4$ and $n=k+3\in \{4,5\}$. For $n=6$, $k=3$, the results in 
\cite{ammann.dahl.humbert:p11b} do not apply directly and just allow to 
prove that 
\[ 
\inf_{c \in [0,1)} 
\mu^{(2)} (\mM_\ci ) \geq   \La^{(1)}_{6,3}.
\]
The case $c=1$ is treated separately: we exploit the fact that $\mM_1$ is 
conformally equivalent to the standard sphere $\mS^6 \setminus \mS^3$ with 
a totally geodesic $3$-sphere removed to show that 
$\mu^{(2)} (\mM_1) \geq \mu(\mS^6) \geq  \La^{(1)}_{6,3}$. We obtain again 
that $\La^{(2)}_{6,3} \geq \La^{(1)}_{6,3}$ (see Appendix \ref{nk=63}).
It remains open whether the same holds for  $n=k+3\geq 7$.
\item
Find lower bounds for $\La^{(1)}_{n,k}$. For this purpose, we show that
$\mu^{(1)} (\mM_\ci )$ can be estimated by the conformal Yamabe constant
of the non-compact manifold~$\mM_\ci$, see Section~\ref{conf_yam}.
We are reduced to find a lower bound for conformal Yamabe constant 
of the product manifold $\mM_\ci$. As mentioned before, there exists
results in this direction; our paper \cite{ammann.dahl.humbert:13} 
gives such a bound if $v \geq 3$ and $w \geq 3$. Also, the work of 
Petean and Ruiz apply if $w = 1$. 
In this paper, we develop a method which completes the remaining 
cases.
\end{itemize}

The technical aspects of the argument in the present paper involve 
symmetrization and stretching maps to relate the the conformal Yamabe
constants of $\mM_\ci$ for different values of $\ci$. This is done in
Section~\ref{section_comparing}.

\begin{remark}
Our methods also apply to find explicit lower bounds for the conformal 
Yamabe constant of $\mH_\ci^v \times (W,h)$, where $(W,h)$ is any compact 
Riemannian manifold, i.e.\ if we replace the round sphere by $(W,h)$. 
The case $(W,h)=\mS^w$ is the only case for which we see applications, 
so for simplicity of presentation we restricted to this case.
\end{remark}

%%%%%%%%%%%%%%%%%%%%%%%%%%%%%%%%%%%%%%%%%%%%%%%%%%%%%%%%%%%%%%%%%%%%%%%%%
\subsection{The generalized Yamabe functional of the model spaces} 
\label{conf_yam}
%%%%%%%%%%%%%%%%%%%%%%%%%%%%%%%%%%%%%%%%%%%%%%%%%%%%%%%%%%%%%%%%%%%%%%%%%

For $u \in C^\infty(\mM_\ci)$, $u \not\equiv 0$, we define the generalized 
Yamabe functional
\[
\cF_\ci^\bi(u)
\definedas
\frac{\int_{\mM_\ci} \left( a_n |du|^2 + \bi u^2 \right) \, dv}
{\|u\|^2_{L^{p_n}(\mM_\ci)}}.
\]
Clearly $\cF_\ci^\bi(u) \geq \cF_\ci^{\bi'}(u)$ if $\bi \geq \bi'$ and 
$\cF_\ci^\bi(u) \geq \frac{\bi}{\bi'} \cF_\ci^{\bi'}(u)$ if 
$0 < \bi \leq \bi'$.

The scalar curvature of $\mM_\ci$ is $s_\ci \definedas \Scal^{G_\ci} =
w(w-1) - \ci^2 v (v-1)$. The conformal Yamabe constant~$\mu_c$ of~$\mM_\ci$
satisfies
\[
\mu_\ci \definedas \mu(\mM_\ci) = \inf \cF_\ci^{s_\ci} (u),
\]
where the infimum is taken over all smooth functions $u$ of compact 
support which do not vanish identically.

If $u$ is a solution of \eqref{eq.conf} as in the definition of 
$\Omega^{(1)}(\mM_\ci)$, then $u$ is $L^2$ by assumption and thus also 
in the Sobolev space $H^{1,2}$. An integration by parts
$\int u \Delta u \, dv = \int |d u|^2 \, dv$ may then be performed in the 
integral defining $\cF_\ci^\bi(u)$, and we conclude that
\[
\mu^{(1)} (\mM_\ci) \geq \mu_\ci. 
\]

Using $\La^{(2)}_{n,k} \geq \La^{(1)}_{n,k}$ and the definition of~$\La^{(1)}_{n,k}$
this implies positive lower bounds for $\La_{n,k}$ for certain pairs $(n,k)$, 
see Table~\ref{fig.analytic}.

%%%%%%%%%%%%%%%%%%%%%%%%%%%%%%%%%%%%%%%%%%%%%%%%%%%%%%%%%%%%%%%%%%%%%%%%%
\subsection{Symmetrization}
%%%%%%%%%%%%%%%%%%%%%%%%%%%%%%%%%%%%%%%%%%%%%%%%%%%%%%%%%%%%%%%%%%%%%%%%%

The group $\Stab_{x_0}(\Isom(\mH_\ci^v))$ of isometries of $\mH_\ci^v$ 
fixing
$x_0\in \mH_\ci^v$ is isomorphic to $O(v)$ and we will fix such an isomorphism 
to identify  $\Stab_{x_0}(\Isom(\mH_\ci^v))$ with $O(v)$.
A function on 
$\mH_\ci^v$ is $O(v)$-invariant if and only if it depends only on the 
distance $d(\cdot, x_0)$ to the point~$x_0$. A function on $\mM_\ci$ is 
$O(v)$-invariant if and only if it depends only on $d(\cdot, x_0)$ and 
the point in $\mS^w$.

\begin{lemma} \label{lemma21}
For each $\ci \in [0,1]$
\[ 
\mu_\ci= \inf \cF_\ci^{s_\ci}(\ti{f}) 
\]
where the infimum is taken over all $O(v)$-invariant functions on 
$\mM_\ci$ which do not vanish identically. 
\end{lemma}

\begin{proof} 
The proof uses standard arguments and we just give a sketch. We must 
show that for any non-negative compactly supported smooth function 
$u: \mM_\ci \to \mR$ there is a $O(v)$-invariant non-negative compactly 
supported smooth function $\tilde{u}: \mM_\ci \to \mR$ satisfying 
$\cF_\ci^{s_\ci}(\tilde{u}) \leq \cF_\ci^{s_\ci}(u)$. If $\phi$ is a 
non-negative function on $\mH_\ci^v$, there is a non-negative 
$O(v)$-invariant function $\phi^*$ defined on the same space called 
the {\em hyperbolic rearrangement of $\phi$}, see \cite{baernstein:94}. 
This has the properties that for $p \geq 1$
\begin{subequations}
\begin{align}
\| \phi^* \|_{L^p( \mH_\ci^v)} &= \| \phi \|_{L^p( \mH_\ci^v)}, 
\label{sph_rea_1} \\
\| \phi_1^* - \phi_2^* \|_{L^p( \mH_\ci^v)} 
&\leq \| \phi_1 - \phi_2 \|_{L^p( \mH_\ci^v)},
\label{contraction} \\
\| d \phi^* \|_{L^p( \mH_\ci^v)} &\leq \| d \phi \|_{L^p( \mH_\ci^v)},
\label{sph_rea_2} 
\end{align}
\end{subequations}
see \cite[Section~4, Corollaries 1 and 3]{baernstein:94}.

Let $u$ be a non-negative function on $\mM_\ci$. 
We set $\tilde{u}(\cdot,y) \definedas (u(\cdot,y))^*$.
From \eqref{sph_rea_1} and \eqref{sph_rea_2} we have
$\| \tilde{u} \|_{L^{p_n}( \mM_\ci)} = \| u \|_{L^{p_n}( \mM_\ci)}$ and 
$\|d_{\mH_\ci^v} \tilde{u} \|_{L^2( \mM_\ci)} \leq
\| d_{\mH_\ci^v} u \|_{L^2( \mM_\ci)}$.
Let $\gamma: (-\ep,\ep) \to S^w$ be a curve. We apply 
\eqref{contraction} with $\phi_1 = u(\cdot,\gamma(t))$, 
$\phi_2 = u(\cdot, \gamma(0))$, divide by $|t|$, and let $t$ 
tend to $0$. From this we conclude 
\[
\| d_{S^w} \tilde{u} (\gamma'(0)) \|_{L^2( \mH_\ci^v \times \{ \gamma(0) \}) } 
\leq \| d_{S^w}u (\gamma'(0)) \|_{L^2( \mH_\ci^v \times \{ \gamma(0) \})}
\]
and $\|d_{S^w} \tilde{u} \|_{L^2( \mM_\ci)} \leq \| d_{S^w} u \|_{L^2( \mM_\ci)}$.
It follows that $\cF_\ci^{s_\ci}(\tilde{u}) \leq \cF_\ci^{s_\ci}(u)$ which 
ends the proof of Lemma \ref{lemma21}. 
\end{proof}

%%%%%%%%%%%%%%%%%%%%%%%%%%%%%%%%%%%%%%%%%%%%%%%%%%%%%%%%%%%%%%%%%%%%%%%%%
\section{Comparing $\cF_\ci^b$ to $\cF_{\ci'}^{b'}$}
\label{section_comparing}
%%%%%%%%%%%%%%%%%%%%%%%%%%%%%%%%%%%%%%%%%%%%%%%%%%%%%%%%%%%%%%%%%%%%%%%%%

We want to estimate $\cF_c^b$ from below in 
terms of $\cF_0^b$ and $\cF_1^{b_1}$ for $b_1$ as large as possible.

%%%%%%%%%%%%%%%%%%%%%%%%%%%%%%%%%%%%%%%%%%%%%%%%%%%%%%%%%%%%%%%%%%%%%%%%%
\subsection{Comparing $\cF_\ci^b$ to $\cF_0^b$}
%%%%%%%%%%%%%%%%%%%%%%%%%%%%%%%%%%%%%%%%%%%%%%%%%%%%%%%%%%%%%%%%%%%%%%%%%

For $\ci \neq 0$ define $\sh{\ci} (t) \definedas \ci^{-1} \sinh(\ci t)$. 
% We also set $\sh{0}(t) = t$. %%%in this paper we never use $\sh{0}$!
In polar coordinates we have
\[
\mH_0^v=\mR^v = ( (0,\infty) \times S^{v-1}, dt^2 + t^2 \rho^{v-1} ),
\]
and
\[
\mH_\ci^v
= ( (0,\infty) \times S^{v-1}, dt^2 + \sh{\ci}(t)^2 \rho^{v-1} ).
\]

\begin{lemma}
For $\ci > 0$ there is a unique diffeomorphism 
$f_\ci: [0,\infty) \to [0,\infty)$ 
for which the map $F_\ci: \mR^v \to \mH_\ci^v$ defined in polar 
coordinates as
\[
F_\ci: (t, \th) \mapsto (f_\ci(t), \th).
\]
is volume preserving. Further $f'_\ci(t) \leq 1$ for all 
$0 \leq t < \infty$. 
\end{lemma}

The map $F_\ci$ squeezes the radial coordinate, so we will call $F_\ci$ 
the \emph{radial squeezing map from $\mR^v$ to $\mH_\ci^v$}.

\begin{proof}
The function 
\[
\phi_\ci(r)
\definedas
\left(\frac{v}{\om_{v-1}}\vol\left(B^{\mH^v_\ci}_{x_0}(r)\right)\right)^{1/v}
=\left(v \int_0^r \sh{\ci}(t)^{v-1} \, dt\right)^{1/v}
\]
is a smooth function $[0,\infty)\to [0,\infty)$. Since $\phi_\ci'(0)=1$, 
$\phi_\ci'(r) > 0$ for $r\geq 0$, and $\lim_{r\to \infty}\phi_\ci(r) = 
\vol(\mH^v_\ci) = \infty$ it is a diffeomorphism. We set 
$f_\ci \definedas \phi_\ci^{-1}$. Let $B_0(r)$ be the ball of radius $r$ 
around $0$ in $\mR^v$. Since $F_\ci$ is assumed to be volume preserving
we have
\[
\vol^{\mR^v} ( B_0(r) )
=
\vol^{\mH_\ci^v } ( F_\ci( B_0(r) ) ),
\]
or 
\begin{equation} \label{f(r)-identity}
\frac{\om_{v-1}}{v} r^v
=
\om_{v-1} \int_0^{f_\ci(r)} \sh{\ci}(t)^{v-1} \, dt.
\end{equation}
Differentiating \eqref{f(r)-identity} we get
\[
r^{v-1}
=
f_\ci'(r) \sh{\ci}( f_\ci(r) )^{v-1}.
\]
{}From \eqref{f(r)-identity} together with 
$\sh{\ci}'(t) = \cosh(\ci t) \geq 1$ we find
\[
\begin{split}
r^v 
&=
\int_0^{f_\ci(r)} v \sh{\ci}(t)^{v-1} \, dt \\
&= 
\int_0^{f_\ci(r)} ( \sh{\ci}(t)^{v} )' \frac{1}{\sh{\ci}'(t)} \, dt \\
&\leq 
\int_0^{f_\ci(r)} ( \sh{\ci}(t)^{v} )' \, dt \\
&=
\sh{\ci}(f(r))^{v},
\end{split}
\]
so $r \leq \sh{\ci}(f_\ci(r))$ and we conclude that $f_\ci'(r) \leq 1$.
\end{proof}

We extend the radial squeezing map to a volume preserving map 
$\wihat F_\ci: \mM_0\to \mM_\ci$ by setting
\[
\wihat F_\ci \definedas F_\ci \times \Id_{\mS^w}: 
\mR^v\times \mS^w\to \mH_\ci^v\times \mS^w.
\]

\begin{prop} \label{proprad}
For $O(v)$-invariant functions $u: \mM_\ci \to \mR$ we have 
\[
\cF_\ci^\bi(u) \geq \cF_0^\bi(u \circ \wihat F_\ci).
\]
\end{prop}

\begin{proof}
The differential $d(u \circ \wihat F_\ci)$ decomposes orthogonally in 
a $\mR^v$-component $d_{\mR^v}(u \circ \wihat F_\ci)$ and 
a $\mS^w$-component $d_{\mS^w}(u\circ \wihat F_\ci)$.
Similarly, $du$ splits orthogonally in a $\mH_\ci^v$-component 
$d_{\mH_\ci^v}u$ and a $\mS^w$-component $d_{\mS^w}u$. Then 
$d_{\mR^v}(u \circ \wihat F_\ci) = d_{\mH_\ci^v}u \circ d\wihat F_\ci$
and $d_{\mS^w}(u \circ \wihat F_\ci) = 
d_{\mS^w}u \circ d\wihat F_\ci = d_{\mS^w}u$. Thus
\[
|d_{\mR^v}(u \circ \wihat F_\ci)| 
= |d_{\mH_\ci^v} u \circ d\wihat F_\ci|
= |d_{\mH_\ci^v} u| f'(t)
\leq |d_{\mH_\ci^v} u|
\]
and
\[
|d_{\mS^w}(u \circ \wihat F_\ci)| = |d_{\mS^w}u|.
\]
It follows that $|d(u \circ \wihat F_\ci)| \leq |du|$. Further the volume 
form is preserved by the map $\wihat F_\ci$ and the Proposition follows.
\end{proof}

\begin{corollary} \label{cor1}
If $s_\ci > 0$ then $\mu_\ci \geq \frac{s_\ci}{s_0}\mu_0$.
\end{corollary}

This corollary gives good estimates if $\ci$ is sufficiently small,
as then $s_\ci>0$. However in case $v>w$ the corollary can no longer 
be applied for $\ci$ close to $1$. 

%%%%%%%%%%%%%%%%%%%%%%%%%%%%%%%%%%%%%%%%%%%%%%%%%%%%%%%%%%%%%%%%%%%%%%%%% 
\subsection{Comparing $\cF_\ci^b$ to $\cF_1^{b_1}$}
%%%%%%%%%%%%%%%%%%%%%%%%%%%%%%%%%%%%%%%%%%%%%%%%%%%%%%%%%%%%%%%%%%%%%%%%%

For $\ci>0$ we define a diffeomorphism $R_\ci: \mH_\ci^v \to \mH_1^v$ by
$R_\ci (t,\theta) = (\ci t,\theta)$. The map $R_\ci$ is a $\ci$-homothety
in the sense that the Riemannian metric of $\mH_\ci^v$ is 
$\eta_\ci^v = \ci^{-2}R_\ci^* \eta_1^v$ where $\eta_1^v$ is the Riemannian 
metric of $\mH_1^v$. Taking the product with the identity map on the round 
sphere we obtain a map $\wihat R_\ci: \mM_\ci \to \mM_1$. The metric $G_\ci$ 
on $\mM_\ci$ is then given by 
$G_\ci = \wihat R_\ci^*(\ci^{-2} \eta_1^v + \rho^w)$.

The following Proposition is an extension of 
\cite[Lemma~3.7]{ammann.dahl.humbert:13b}. 

\begin{proposition} \label{propsig}
If $\ci \in (0,1)$, then 
$\cF_\ci^{\ci^2 s_1}(u\circ \wihat R_\ci) \geq \ci^{2w/n} \cF_1^{s_1}(u)$ 
for all functions $u \in C^{\infty}_c (\mM_1)$. 
\end{proposition} 

\begin{proof}
We have
\[
\begin{split}
|d(u \circ \wihat R_\ci)|^2_{G_\ci} 
&= |R_\ci^*(du)|_{G_\ci}^2 \\
&= |du|_{\ci^{-2} \eta_1^v+\rho^w}^2 \\
&= \ci^2 |d_{\mH_\ci^v} u|^2_{\eta^v_1} + |d_{\mS^w} u|_{\rho^w}^2 \\
&\geq \ci^2 \left(|d_{\mH_\ci^v} u|^2_{\eta^v_1} + |d_{\mS^w} u|_{\rho^w}^2\right) \\
&= \ci^2 |du|^2_{g_1}.
\end{split} 
\]
In addition, $dv^{G_\ci} = \ci^{-v} \wihat R_\ci^* dv^{g_1}$. From this we 
find that 
\[
\begin{split}
\cF_\ci^{\ci^2 s_1}(u \circ \wihat R_\ci)
&= 
\frac{
\int_{\mM_c} \left( a_n |d(u \circ \wihat R_\ci)|^2_{G_\ci} 
+ \ci^2 s_1 (u \circ \wihat R_\ci)^2\right) \, dv^{G_\ci}}
{{\left(\int_{\mR^v \times S^w} (u \circ \wihat R_\ci)^{p_n} 
\, dv^{G_\ci}\right)}^{\frac{2}{p_n}}} \\
&\geq 
\frac{\int_{\mM_1} 
\left( a_n\ci^2 |du|^2_{g_1} + \ci^2 s_1 u^2\right) \ci^{-v} \, dv^{g_1}}
{{\left(\int_{\mR^v \times S^w} u^{p_n} \ci^{-v} \, dv^{g_1}
\right)}^{\frac{2}{p_n}}}\\
&=
\ci^{2w/n} \cF_1^{s_1}(u),
\end{split} 
\]
which is the statement of the Proposition.
\end{proof} 

To apply the proposition, note that 
\[
s_\ci = w(w-1) - \ci^2 v (v-1) 
\geq \ci^2(w(w-1) - v (v-1)) = \ci^2 s_1.
\]
This implies 
\[
\cF_\ci^{s_\ci}(u\circ \wihat R_\ci) \geq \cF_\ci^{\ci^2 s_1}(u).
\]
By taking the infimum over all non-vanishing smooth functions 
$u\in C^\infty_c(\mM_1)$ with compact support we obtain the following.
\begin{corollary}\label{cor2}
For $c\in (0,1)$ we have 
\[
\mu_c \geq c^{2w/n} \mu_1 .
\]
\end{corollary}

This estimate gives uniform estimates fur $\mu_c$ if $c$ is bounded 
away from $0$. Because of $\mu_1 = \mu(\mS^n)$ we obtain explicit bounds 
in any dimension. However these bounds tend to $0$ as $c\to 0$.

%%%%%%%%%%%%%%%%%%%%%%%%%%%%%%%%%%%%%%%%%%%%%%%%%%%%%%%%%%%%%%%%%%%%%%%%% 
\section{Conclusions}
%%%%%%%%%%%%%%%%%%%%%%%%%%%%%%%%%%%%%%%%%%%%%%%%%%%%%%%%%%%%%%%%%%%%%%%%%

%%%%%%%%%%%%%%%%%%%%%%%%%%%%%%%%%%%%%%%%%%%%%%%%%%%%%%%%%%%%%%%%%%%%%%%%%
\subsection{Interpolation of the previous inequalities}
%%%%%%%%%%%%%%%%%%%%%%%%%%%%%%%%%%%%%%%%%%%%%%%%%%%%%%%%%%%%%%%%%%%%%%%%%

We now improve the bounds obtained in Corollaries \ref{cor1} and 
\ref{cor2} by combining Propositions \ref{proprad} and \ref{propsig} in 
an interpolation argument. 

\begin{theorem}\label{theo.concl}
For all $c\in (0,1)$ we have
\begin{equation} \label{est.general}
\mu_\ci
\geq 
\left( 
\frac{\mu_0}{\mu_1} - 
\frac{\ci^2 v(v-1)}{(1-\ci^2)w(w-1) + \ci^2 v(v-1)}
\left(\frac{\mu_0}{\mu_1} -\ci^{2w/n}\right)
\right) \mu_1 
\end{equation}
and 
\begin{equation}\label{est.sa}
\mu_\ci \geq \ci^{2w/n} \mu_1.
\end{equation}
\end{theorem}
As discussed in Appendix~\ref{app.opti}, Inequality~\eqref{est.general} 
is stronger than Inequality~\eqref{est.sa} for $\ci^{2w/n} < \mu_0/\mu_1$
and Inequality~\eqref{est.sa} is stronger for $\ci^{2w/n} > \mu_0/\mu_1$.

\begin{proof}
Inequality~\eqref{est.sa} is the statement of Corollary~\ref{cor2}. 
Assume that $\la \geq 0$ and $\tau \geq 0$ satisfy 
\begin{align}
\la + \tau &\leq 1 , \label{cond1} \\
\la \ci^2 s_1 + \tau s_0 &\leq s_\ci . \label{cond2}
\end{align}
Then we get
\begin{equation} \label{estimate-la-tau}
\begin{split}
\cF_\ci^{s_\ci}(u) 
&\geq
\la \cF_\ci^{\ci^2 s_1}(u\circ \wihat R_\ci^{-1}) 
+ \tau \cF_\ci^{s_0}(u\circ \wihat F_\ci)\\
&\geq
\la \ci^{2w/n} \cF_1^{s_1}(u\circ \wihat R_\ci^{-1}) 
+ \tau \cF_0^{s_0}(u\circ \wihat F_\ci)\\
&\geq
\la \ci^{2w/n} \mu_1 + \tau \mu_0, 
\end{split}
\end{equation}
where we used Proposition~\ref{propsig} for the second inequality. 
It follows that 
\begin{equation} \label{estimate-la-tau-2}
\mu_\ci \geq \la \ci^{2w/n} \mu_1 + \tau \mu_0.
\end{equation}
The lines described by $\la + \tau = 1$ and 
$\la \ci^2 s_1 + \tau s_0 = s_\ci$ intersect in $(\la_0,\tau_0)$ where
\begin{equation} \label{la0-eq}
\lambda_0 = \frac{v(v-1)}{(\ci^{-2}-1)w(w-1) + v (v-1)} \in (0,1),
\qquad 
\tau_0 = 1 - \la_0,
\end{equation}
see Appendix~\ref{app.opti}. Setting $\la \definedas \la_0$ and 
$\tau \definedas \tau_0$ in \eqref{estimate-la-tau-2} 
yields Inequality~\eqref{est.general}.
%
% \footnote{Removed stuff:
% In Appendix \ref{app.opti} we find the optimal values of $\la$ and
% $\tau$ for this inequality. The result is that the optimal estimate is
% obtained in $(1,0)$ (for $\ci^{2w/n} \geq \mu_0/\mu_1$) or in 
% $(\la_0,\tau_0)$ (for $\ci^{2w/n} \leq \mu_0/\mu_1$). Here
% $(\la_0,\tau_0)$ is the solution of \eqref{cond1}-\eqref{cond2} with
% inequalities replaced by equalities. We obtain 
% \begin{equation} %\label{est.general}
% \mu_\ci
% \geq 
% \left( 
% \left(\ci^{2w/n}-\frac{\mu_0}{\mu_1} \right)
% \frac{v(v-1)}{(\ci^{-2}-1)w(w-1) + v(v-1)}
% + \frac{\mu_0}{\mu_1}
% \right)\mu_1, 
% \end{equation}
% if $\ci^{2w/n} \leq \mu_0/\mu_1$, and 
% \[
% \mu_\ci \geq \ci^{2w/n} \mu_1
% \]
% for $\ci^{2w/n} \geq \mu_0/\mu_1$.}
\end{proof}

The estimates obtained by the theorem rely on explicit lower bounds
for $\mu_0$. Such lower bounds can be found in the literature in the 
following cases.

% changed to (i), (ii),... standard numbering in enumerate 
% conflicts with equation numbering (1), (2),...
\renewcommand{\theenumi}{\roman{enumi}}
\renewcommand{\labelenumi}{(\theenumi)}
\begin{enumerate}
\item \label{caseone} 
$v=1$, $w\geq 2$. Then $\mu_0 = \mu_1 = \mu_c = \mu(\mS^n)$ for all 
$c \in (0,1)$. This case is trivial as $\mR \times \mS^w$ is conformal 
to a round sphere of dimension $n=w+1$ with two points removed.
\item \label{casetwo} 
$(v,w)\in\{(2,2),(2,3),(2,7),(2,8),(3,2)\}$. In these cases bounds 
have been derived in \cite{petean.ruiz:11,petean.ruiz:13} using 
isoperimetric profiles.
\item \label{casethree} 
$v \geq 3$ and $w \geq 3$. See \cite{ammann.dahl.humbert:13} where 
an explicit lower bound of the Yamabe functional of $\mR^v \times \mS^w$ 
in terms of the Yamabe functionals of $\mR^v$ and $\mS^w$ is used. 
\item \label{casefour} 
$v \geq 4$ and $w =2$. This case is not explicitly written in 
\cite{ammann.dahl.humbert:13} but can be deduced from the main 
result of that paper. We just observe that this result implies that 
\[
\mu_0 \geq 
\frac{n a_n}{ (3 a_3)^{\frac{3}{n}} ((n-3) a_{n-3})^{\frac{n-3}{n}}} 
\mu(\mR^{n-3})^{\frac{n-3}{n}} 
\mu(\mR \times \mS^2)^{\frac{3}{n}}
\]
where $a_k \definedas \frac{4(k-1)}{k-2}$ for $k \geq 3$. Next, note 
that $\mu(\mR^{n-3}) = \mu(\mS^{n-3})$ and since $\mR \times \mS^2$ is 
conformally equivalent to $\mS^3$ with two points removed we have 
$\mu(\mR \times S^2) = \mu(\mS^3)$. Hence, we get 
\[
\mu_0 \geq 
\frac{n a_n}{ 24^{\frac{3}{n}} ((n-3) a_{n-3})^{\frac{n-3}{n}}} 
\mu(\mS^{n-3})^{\frac{n-3}{n}} 
\mu(\mS^3)^{\frac{3}{n}}.
\] 
In the case $(v,w)=(4,2)$ this leads to
\begin{equation}\label{ineq.gamma42}
\mu_0\geq 0.56885 \mu_1>54.77.
\end{equation}
A similar argument also yields lower bounds for $\mu_0$ in the cases
$v-2 \geq w \geq 3$. These bounds on $\mu_0$ are 
slightly stronger than the ones in \eqref{casethree}.
\end{enumerate}

The estimate is optimal in Case~\eqref{caseone}. In this case 
nothing remains to be proven, and we will not discuss it further.
In Cases~\eqref{casetwo} and~\eqref{casethree} the bound is not likely 
to be optimal. Any improvement of the lower bound for $\mu_0$ would 
improve the bounds obtained in Theorem~\ref{theo.concl}. In 
\cite{ammann.dahl.humbert:13} a lower bound on $\mu_\ci$ is derived 
which is uniform in $\ci$. Thus Theorem~\ref{theo.concl} does not 
currently yield improved estimates in Case~\eqref{casethree}. However, 
if a better lower bound for $\mu_0$ is available, it might be relevant as 
well, and will be also considered in the following. The most important 
applications thus come in Case~\eqref{casetwo}.

%%%%%%%%%%%%%%%%%%%%%%%%%%%%%%%%%%%%%%%%%%%%%%%%%%%%%%%%%%%%%%%%%%%%%%%%%
\subsection{Analytical Conclusions}
%%%%%%%%%%%%%%%%%%%%%%%%%%%%%%%%%%%%%%%%%%%%%%%%%%%%%%%%%%%%%%%%%%%%%%%%%

We now want to derive concrete bounds on $\La_{v+w,v-1}$ for special 
values of $v$ and $w$.

\begin{corollary} \label{cor52}
For all $c\in [0,1]$ and all $v\geq 2$ and $w\geq 2$ we obtain
\begin{equation} \label{formula.large.w}
\mu_\ci
\geq 
\left(
1 - \frac{ v(v-1) }
{ \left( \sqrt{v(v-1)} + \sqrt{w(w-1)} \right)^2 }
\right)\mu_0. 
\end{equation}
\end{corollary}

\begin{proof}
Using \eqref{est.general} and the facts that $\mu_1 > \mu_0$ and 
$\ci^{2w/n} \geq \ci^2$ we deduce
\begin{equation}\label{ineq.anal}
\mu_\ci
\geq 
\left(
1- \frac{(1 - \ci^2) \ci^2 v (v-1)}{(1-\ci^2) w(w-1) + \ci^2 v(v-1)}
\right)\mu_0 
\end{equation}
for general values of $v$ and $w$. The right hand side attains its
minimum over $\ci \in [0,1]$ for
\[
\ci^2 = \frac{ \sqrt{w(w-1)} }{ \sqrt{v(v-1)} + \sqrt{w(w-1)} }, 
\]
from which \eqref{formula.large.w} follows.
\end{proof}

\begin{example} 
$v=2$, $w=3$: In \cite[Theorem~1.4]{petean.ruiz:13} Petean and Ruiz 
have obtained $\mu(\mR^2\times \mS^3) \geq 0.75 \mu(\mS^5)$, that is
$\mu_0\geq 0.75 \mu_1$. Using \eqref{formula.large.w} we obtain 
\[
\mu_\ci\geq \frac{\sqrt{3}}2 \mu_0 \geq 0.649 \mu_1 \geq 51.2
\]
Thus $\Lambda_{5,1}\geq 51.2$. 

Compare this value with $\mu(\mS^5)=78.996...$
\end{example}

\begin{example} %\label{ex.ana.2} 
$v=2$, $w=7$: In \cite[Theorem~1.6]{petean.ruiz:13} Petean and Ruiz 
have obtained $\mu(\mR^2\times \mS^7)\geq 0.747 \mu(\mS^9)$, that is
$\mu_0\geq 0.747 \mu_1$. Using \eqref{formula.large.w} we obtain 
\[
\mu_\ci
\geq \left( 1 - \frac{2}{(\sqrt{2}+\sqrt{42})^2}\right) \mu_0 
\geq 0.723 \mu_1 \geq 106.9
\]
Thus $\Lambda_{9,1}\geq 106.9$
%$\Lambda_{9,1} \geq 106.9145822$. 

Compare this value with $\mu(\mS^9)=147.87...$
\end{example}

\begin{example} %\label{ex.ana.3} 
$v=2$, $w=8$: In \cite[Theorem~1.6]{petean.ruiz:13} Petean and Ruiz 
have obtained $\mu(\mR^2\times \mS^8) \geq 0.626 \mu(\mS^{10})$, that is 
$\mu_0\geq 0.626 \mu_1$. 
Using \eqref{formula.large.w} we obtain 
\[
\mu_\ci\geq \left( 1 - \frac{2}{(\sqrt{2}+\sqrt{56})^2}\right) \mu_0 
\geq 0.610 \mu_1 \geq 100.69
\]
Thus $\Lambda_{10,1}\geq 100.69$.
%$\Lambda_{10,1} \geq 100.692$. 

Compare this value with $\mu(\mS^{10})=165.02...$
\end{example}

%Example: $v=w=2$, leads to $\La_{4,1}$.
%And thus $\frac{\mu_\ci}{\mu_0} -1\geq -\frac14$. This implies
%$\mu_\ci\geq \frac34 \mu_0$, and thus $\La_{4,1}\geq \frac34 \mu_0$. 
%Using Petean this gives $\La_{4,1}\geq 0.51 \mu(\mS^4)$.

\begin{figure}
\begin{tabular}{c|c||l|l|l|l}
$(v,w)$ & $(n,k)$  & $\mu_0/\mu_1$ & Analytic & Numeric & $\mu_1=\mu(\mS^n)$\\
\hline\hline
$(2,2)$ & $(4,1)$  & $0.68$  & $\phantom{1}38.9$ & $\phantom{1}38.9$ & $\phantom{1}61.56$\\
$(2,3)$ & $(5,1)$  & $0.75$  & $\phantom{1}51.2$ & $\phantom{1}56.6$ & $\phantom{1}79.00$\\
$(2,7)$ & $(9,1)$  & $0.747$ & $106.9$           & $109.2$ & $147.87$\\
$(2,8)$ & $(10,1)$ & $0.626$ & $100.6$           & $102.6$ & $165.02$\\
$(3,2)$ & $(5,2)$  & $0.63$  & $\phantom{1}29.7$ & $\phantom{1}45.1$ & $\phantom{1}79.00$\\
$(4,2)$ & $(6,3)$  & $0.568$ & $\phantom{1}36.4$ & $\phantom{1}49.9$ & $\phantom{1}96.30$\\
%More exact $(2,3)$ & $0.75$ & $51.2$ & $56.6199569$ & $78.99686$\\
%More exact $(3,2)$ & $0.63$ & $29.7$ & $45.136$ & $78.99686$\\
%More exact $(2,7)$ & $0.747$ & & $109.2$ & $147.8778$\\
%More exact $(2,8)$ & $0.626$ & & $102.6$ & $165.0220642$\\
\end{tabular}
\caption{Lower estimates for~$\inf\mu_\ci=\Lambda_{n,k}$. 
The fourth column shows 
the analytic estimates from Corollary~\ref{cor52} and~\ref{cor54}. 
The fifth column shows the numerical estimates from 
Subsection~\ref{subsec.num}. The value for~$\mu_1$ is approximate,
whereas the lower bounds are rounded down.}\label{fig.analytic}
\end{figure}

In the case $v=w$ we find better estimates for the right hand side of 
\eqref{est.general}.

\begin{corollary} \label{cor54}
Assume $v = w \geq 2$ and $\mu_0/\mu_1 \geq \ga > 0$. Then
\[
\inf_{\ci\in[0,1]} \mu_\ci 
\geq \left( \ga -\frac{4}{27} \ga^3 \right) \mu_1 
\]
\end{corollary}

\begin{proof}
Using $v=w$ we obtain directly from~\eqref{est.general}: 
\[
\mu_\ci 
\geq 
\left(\left( \ci - \frac{\mu_0}{\mu_1} \right) 
\frac{1}{\ci^{-2}}+ \frac{\mu_0}{\mu_1} \right) \mu_1
= 
\left( \ci^3 - \ci^2 \frac{\mu_0}{\mu_1} + \frac{\mu_0}{\mu_1} \right)\mu_1
\geq 
\left( \ci^3 - \ci^2 \ga + \ga \right)\mu_1
\]
for any $\ga\in(0,\mu_0/\mu_1]$. On the interval $[0,1]$ the right hand 
side attains its minimum in $\ci = \frac{2}{3} \ga$. 
%Let $\ga$ be a positive lower bound for $\frac{\mu_0}{\mu_1}$, 
This yields the statement of the corollary.
%We conclude 
%\[
%\inf_\ci \mu_\ci 
%\geq \left( \ga -\frac{4}{27} \ga^3 \right) \mu_1 
%\]
%if $\ga>0$ is a lower bound for $\mu_0/\mu_1$.
\end{proof}

\begin{example}
For $v=w=2$ Petean and Ruiz \cite[Theorem~1.2]{petean.ruiz:11} have 
derived the bound $\ga=0.68$. This yields
\[
\La_{4,1} \geq 0.63 \mu_1 \geq 38.9.
\]
\end{example}

In the case $v>w$ one can use $\ci^{2w/n} \geq \ci$ which improves 
inequality\eqref{ineq.anal} to  
\begin{equation*}
\mu_\ci
\geq 
\left(
1- \frac{(1 - \ci) \ci^2 v (v-1)}{(1-\ci^2) w(w-1) + \ci^2 v(v-1)}
\right)\mu_0 
\end{equation*}
which again yields better estimates for the right hand side of 
\eqref{est.general}.

Obviously in the case $(v,w)=(4,2)$ the determination of the value $c$ for
which $\mu_c$ is minimal, gives the equation
$5c^3+3c=2$ which has as only real solution 
\[
c = \frac15 \sqrt[3]{25+5\sqrt{30}} + \frac1{\sqrt[3]{25+5\sqrt{30}}}
\approx 0.48108.
\]

\begin{example}
For $(v,w)=(4,2)$ we have derived the bound $\ga=0.56885$, see 
equation~\eqref{ineq.gamma42}. This yields  
\[
\La_{6,3} \geq 0.3788 \mu_1 \geq 36.4
\]
\end{example}

The explicit values deduced from the above corollaries
are summarized in Table~\ref{fig.analytic}.

%%%%%%%%%%%%%%%%%%%%%%%%%%%%%%%%%%%%%%%%%%%%%%%%%%%%%%%%%%%%%%%%%%%%%%%%%
\subsection{Numerical Conclusions}\label{subsec.num}
%%%%%%%%%%%%%%%%%%%%%%%%%%%%%%%%%%%%%%%%%%%%%%%%%%%%%%%%%%%%%%%%%%%%%%%%%

Numerical computations yield better bounds. Such improved bounds are 
important for applications, especially for some particular values, as 
for example the case $v=3$, $w=2$.

Using the procedure ``Minimize'' from the ``Optimization'' package
of the program Maple 13.0 we numerically minimized the right hand side 
of \eqref{est.general}. The results of this calculation provided the
bounds given in the column ``Numeric'' of Table~\ref{fig.analytic}.

%\begin{figure}
%The numeric table!!!
%%\begin{tabular}
%%\end{tabular}
%\caption{The numeric table!!!}\label{fig.numeric}
%\end{figure}

\begin{example}
Assume $v=3$ and $w=2$. In \cite[Theorem~1.4]{petean.ruiz:13} Petean 
and Ruiz have obtained $\mu(\mR^3\times \mS^2)\geq 0.63 \mu(\mS^5)$, 
that is $\mu_0\geq 0.63 \mu_1$. A numerical evaluation of 
\eqref{est.general} yields 
\[
\inf_{\ci\in [-1,1]}\mu_\ci\geq 0.571 \mu_1 > 45.1,
\]
and we conclude that $\La_{5,2} > 45.1$.
\end{example}

\begin{example} \label{ex.num.2}
Assume $v=2$ and $w=7$. In \cite[Theorem~1.6]{petean.ruiz:13} Petean 
and Ruiz have obtained $\mu(\mR^2\times \mS^7)\geq 0.747 \mu(\mS^9)$, 
that is $\mu_0\geq 0.747 \mu_1$. A numerical evaluation of 
\eqref{est.general} yields 
%\[
%\inf_{\ci\in [-1,1]}\mu_\ci\geq 0.739076588197446038 \mu_1 > 109.2930723,
%\]
\[
\inf_{\ci\in [-1,1]}\mu_\ci\geq 0.739 \mu_1 > 109.2,
\]
and we conclude that $\La_{9,1} > 109.2$.
\end{example}

\begin{example} \label{ex.num.3}
Assume $v=2$ and $w=8$. In \cite[Theorem~1.6]{petean.ruiz:13} Petean 
and Ruiz have obtained $\mu(\mR^2\times \mS^8)\geq 0.626 \mu(\mS^{10})$, 
that is $\mu_0\geq 0.626 \mu_1$. A numerical evaluation of 
\eqref{est.general} yields 
%\[
%\inf_{\ci\in [-1,1]}\mu_\ci\geq 0.622295802175962920 \mu_1 > 102.6925378
%\]
\[
\inf_{\ci\in [-1,1]}\mu_\ci\geq 0.622 \mu_1 > 102.6
\]
and we conclude that 
%$\La_{10,1} > 102.6925378$.
$\La_{10,1} > 102.6$.
\end{example}

\begin{example} \label{ex.num.42}
Assume $v=4$ and $w=2$. In~\eqref{ineq.gamma42}
we have seen that $\mu(\mR^4\times \mS^2)\geq 0.56885 \mu(\mS^{6})$, 
that is $\mu_0\geq 0.56885 \mu_1$. A numerical evaluation of 
\eqref{est.general} yields 
%\[
%\inf_{\ci\in [-1,1]}\mu_\ci\geq 0.519097200943746672 \mu_1 > 49.98764(8)
% attained for $\ci= 0.248664556665240344$.
%\]
\[
\inf_{\ci\in [-1,1]}\mu_\ci\geq 0.51909 \mu_1 > 49.98
\]
and we conclude that 
%$\La_{10,1} > 102.6925378$.
$\La_{10,1} > 102.6$.
\end{example}

Similar bounds for other dimensions could also be obtained using the 
same method. We will see that the cases derived as examples above have 
interesting topological applications.

%%%%%%%%%%%%%%%%%%%%%%%%%%%%%%%%%%%
\subsection{Bibliographic remark}
%%%%%%%%%%%%%%%%%%%%%%%%%%%%%%%%%%%%
At the time when this article went into press, there was important progress
connected to the Yamabe constant $\mu_c=\mu(\mM_c)$: Solutions of the 
Yamabe equation on $\mM_c$ which are constant on the sphere component, 
were studied systematically 
in \cite{henry.petean:p13}. 

%%%%%%%%%%%%%%%%%%%%%%%%%%%%%%%%%%%%%%%%%%%%%%%%%%%%%%%%%%%%%%%%%%%%%%%%%
\section{Topological applications}
\label{section_top_appl}
%%%%%%%%%%%%%%%%%%%%%%%%%%%%%%%%%%%%%%%%%%%%%%%%%%%%%%%%%%%%%%%%%%%%%%%%%

The lower bounds for $\La_{n,1}$, $n\in\{4,5,9,10\}$, 
and $\La_{5,2}$ and $\La_{6,3}$
lead to estimates of the Yamabe invariant for certain 
classes of manifolds.

%%%%%%%%%%%%%%%%%%%%%%%%%%%%%%%%%%%%%%%%%%%%%%%%%%%%%%%%%%%%%%%%%%%%%%%%%
\subsection{Applications of the lower bound for $\La_{5,2}$}
%%%%%%%%%%%%%%%%%%%%%%%%%%%%%%%%%%%%%%%%%%%%%%%%%%%%%%%%%%%%%%%%%%%%%%%%%

The following two propositions are standard consequences of the methods 
developed for the proof of the h-cobordism theorem. A proof for a similar
statement can be found in 
\cite[Theorem IV.4.4, pages 299--300]{lawson.michelsohn:89}. As we do not 
know of a reference for the formulations given here we include their proofs.

\begin{proposition} \label{prop.spin}
Let $M_0$ and $M_1$ be non-empty, compact, connected, and simply connected 
spin manifolds of dimension $n\geq 5$. Assume that $M_0$ and $M_1$ are 
spin bordant. Then one can obtain $M_1$ from $M_0$ by a sequence of 
surgeries of dimensions $\ell$ where $2 \leq \ell \leq n-3$. 
\end{proposition}

\begin{proof}
Let $W$ be a spin bordism from $M_0$ to $M_1$. By surgeries in the interior 
we simplify $W$ to be connected, simply connected, and have $\pi_2(W) = 0$ 
(one then says $W$ is $2$-connected). Then $H_i(W,M_j) = 0$ for $i=0,1,2$. 
We apply \cite[VIII Thm.~4.1]{kosinski:93} for $k=3$ and $m=n+1$. One 
obtains that there is a handle presentation of the bordism such that for 
any $i < 3$ and any $i>n-2$ the number of $i$-handles is given by 
$b_i(W,M_0)$. Any $i$-handle corresponds to a surgery of dimension $i-1$. 
It remains to show that $b_i(W,M_0)=0$ for $i \in \{0,1,2,n+1,n,n-1\}$. 
This is trivial for $i\in \{0,1,2\}$. By Poincar\'e duality 
$H^{n+1-i}(W,M_0)$ is dual to $H_i(W,M_1)$ which vanishes for $i=0,1,2$. 
On the other hand the universal coefficient theorem tells us that the 
free parts of $H^i(W,M_0)$ and $H_i(W,M_0)$ are isomorphic. Thus 
$b_i(W,M_0)$ which is by definition the rank of (the free part of) 
$H_i(W,M_0)$ vanishes for $i \in \{n+1,n,n-1\}$.
\end{proof}

\begin{proposition} \label{prop.non-spin}
Let $M_0$ and $M_1$ be non-empty compact connected and simply connected 
non-spin manifolds of dimension $n\geq 5$, and assume that these manifolds 
are oriented bordant. Then one can obtain $M_1$ from $M_0$ by a sequence 
of surgeries of dimensions $\ell$, $2 \leq \ell \leq n-3$. 
\end{proposition}

\begin{proof}
The proof is similar to the proof in the spin case. However the bordism 
$W$ cannot be simplified to $\pi_2(W) = 0$, but only to 
$\pi_2(W)=\mZ/2\mZ$ with surjective maps $\pi_2(M_j)\to \pi_2(W)$. This 
implies again that $H_i(W,M_j) = 0$ for $i=0,1,2$, and $j=1,2$. The proof 
continues exactly as in the spin case.
\end{proof}

\begin{corollary}\label{cor.dim.5}
Let $M$ be a compact simply connected manifold of dimension $5$, then
\[
45.1 < \si(M) \leq \mu(\mS^5) < 79.
\]
\end{corollary}

\begin{proof}
The upper bound for $\si(M)$ is standard.

To prove the lower bound we consider first the case when $M$ is spin.
As the $5$-dimensional spin bordism group $\Omega_5^{\rm Spin}$ is trivial,
$M$ is the boundary of a compact $6$-dimensional spin manifold. By 
removing a ball we obtain a spin bordism from $S^5$ to $M$. Using 
Proposition~\ref{prop.spin} we see that $M$ can be obtained by 
$2$-dimensional surgeries from $S^5$. As a consequence 
$\si(M) \geq \La_{5,2} > 45.1$. 

Next we consider the case when $M$ is not spin. The oriented bordism 
group $\Omega_5^{\rm SO}$ is isomorphic to $\mZ/2\mZ$, and the Wu manifold 
$\SU(3)/\SO(3)$ represents a non-trivial element in $\Omega_5^{\rm \SO}$. 
Thus $M$ is either oriented bordant to the empty set or to $\SU(3)/\SO(3)$.

We consider now the case that $M$ is oriented bordant to $\SU(3)/\SO(3)$.
By Appendix~\ref{app.wu.mfd} we see that $\si(\SU(3)/\SO(3)) > 64$. 
Since $\SU(3)/\SO(3)$ is not spin Proposition~\ref{prop.non-spin} implies 
that we can obtain $M$ from $\SU(3)/\SO(3)$ by a finite number of 
$2$-dimensional surgeries. Thus 
\[
\si(M) \geq \min \left( \La_{5,2}, \si(\SU(3)/\SO(3)) \right) > 45.1.
\]

It remains to consider the case that $M$ is oriented bordant to the 
empty set, or equivalently to $S^5$. However, $S^5$ is spin and cannot 
be used to apply Proposition~\ref{prop.non-spin}. Instead we use the space
$\SU(3)/\SO(3)\# \SU(3)/\SO(3)$ which is simply connected, non-spin 
and an oriented boundary. By \cite[Theorem~2]{kobayashi:87} we know 
that $\si(\SU(3)/\SO(3) \# \SU(3)/\SO(3)) \geq \si(\SU(3)/\SO(3))$. 
We apply Proposition~\ref{prop.non-spin} with
$M_0 = \SU(3)/\SO(3) \# \SU(3)/\SO(3)$ and $M_1 = M$ and thus we obtain
$M$ from $M_0$ by a finite number of $2$-dimensional surgeries. From 
this we find
\[
\si(M) \geq \min \left(\La_{5,2}, \si(\SU(3)/\SO(3))\right) > 45.1
\]
which concludes the proof of the corollary.
\end{proof}

Let us compare the lower bound $45.1$ for simply connected 
$5$-manifolds to the expected values for the smooth Yamabe invariant on 
non-simply-connected spherical space forms in dimension $5$. Assume that 
$M = S^5 / \Gamma$ where the finite group $\Gamma \subset \SO(6)$ 
acts freely on $S^5$. It was conjectured by 
Schoen~\cite[Page 10, lines 6--11]{schoen:89} that on such manifolds 
the supremum in the definition of the smooth Yamabe number is attained 
by the standard conformal structure. If this is true, then $\si(\mR P^5)$ 
would be equal to $45.371\dots$. Except $S^5$ and $\mR P^5$ all 
$5$-dimensional space forms would have $\si$-invariant below $45.1$.

%%%%%%%%%%%%%%%%%%%%%%%%%%%%%%%%%%%%%%%%%%%%%%%%%%%%%%%%%%%%%%%%%%%%%%%%%
\subsection{Applications of the lower bound for $\La_{6,3}$}
%%%%%%%%%%%%%%%%%%%%%%%%%%%%%%%%%%%%%%%%%%%%%%%%%%%%%%%%%%%%%%%%%%%%%%%%%

\begin{corollary}\label{cor.dim.6}
Let $M$ be a compact simply connected manifold of dimension $6$, then
\[
49.9 < \si(M) \leq \mu(\mS^6) < 96.30.
\]
\end{corollary}

\begin{proof}
The proof of this corollary is a straightforward adaptation of the proof of 
previous corollary, using the fact that both the spin bordism group and the
oriented bordism group are trivial in dimension $6$.
We obtain
\[
\si(M) \geq \min(\La_{6,2},\La_{6,3}) \geq 49.9.
\]
\end{proof}

%%%%%%%%%%%%%%%%%%%%%%%%%%%%%%%%%%%%%%%%%%%%%%%%%%%%%%%%%%%%%%%%%%%%%%%%%
\subsection{Applications of the lower bound for $\La_{9,1}$ 
and $\La_{10,1}$ to spin manifolds}
%%%%%%%%%%%%%%%%%%%%%%%%%%%%%%%%%%%%%%%%%%%%%%%%%%%%%%%%%%%%%%%%%%%%%%%%%

For a compact spin manifold $M$ of dimension $n$ the alpha-genus 
$\alpha(M) \in KO_n$ is equal to the index of the Clifford-linear Dirac 
operator on $M$. It depends only on the spin bordism class of $M$.

\begin{lemma} \label{lemma9.10}
Let $M$ be a compact $2$-connected spin manifold of dimension 
$n \in \{9,10\}$ which has $\al(M) = 0$. Then $M$ is obtained from $S^9$ 
or $\mH P^2 \times S^1$ (for $n=9$) or from $S^{10}$ or 
$\mH P^2\times S^1 \times S^1$ (for $n=10$) by a sequence of surgeries 
of dimensions $k \in \{0,1,\ldots,n-4\}$. All these surgeries are 
compatible with orientation and spin structure.
\end{lemma}

Note that $S^1$ carries two spin structure. One spin structure
is obtained from the spin structure on $D^2$ by restriction to the 
boundary $S^1 = \partial D^2$, and it is called the
\emph{bounding spin structure}. In the above lemma we assume that all
manifolds $S^1$ are equipped with the other spin structure, the
\emph{non-bounding spin structure}.

\begin{proof}
%It follows from \cite[Theorem~B]{stolz:92} that a manifold $M$ as above 
%is spin bordant to the total space $P$ of an $\mH P^2$-bundle over a 
%closed spin manifold $B$ of dimension $k=n-8$, and the structure group 
%of the bundle is $\PSp(3)$. The fact that $\Sp(3)$ is connected and 
%simply connected together with the exact sequence 
%$1 \to \Sp(1) \cong S^3 \to \Sp(3) \to \PSp(3) = \Sp(3)/\Sp(1) \to 1$ 
%shows that $\PSp(3)$ is connected and simply connected.
%
%In the case $n=9$, the connectedness of $\PSp(3)$ implies that $P$ can 
%be chosen as a product of $\PSp(3)$ and $B$. Furthermore $B$ is a disjoint 
%union of circles equipped with a spin structure. If any of these circles 
%is a spin-boundary, it can be removed. The disjoint union of two circles 
%with non-bounding spin structure is a spin boundary since it bounds 
%$S^1 \times [0,1]$. We thus can assume $B = \emptyset$ or $B = S^1$ with 
%the non-bounding spin-structure, and $P = \mH P^2 \times B$.
%
%In the case $n=10$, the base $B$ is a surface with spin structure.
%As the structure group is connected and simply connected, $P$ is 
%spin-diffeomorphic to a product. In dimension $2$ there are two 
%spin-bordism classes. The non-trivial class is generated by 
%$S^1 \times S^1$ with non-bounding spin structures on both factors 
%$S^1$. Thus $M$ is bordant to $P = \emptyset$ or to 
%$P = \mH P^2 \times S^1 \times S^1$.
%
From the description of the Spin bordism group in 
\cite{anderson.brown.peterson:66} and 
\cite{anderson.brown.peterson:67} we know that $M$ is spin bordant 
to $P = \emptyset$ or to $P = \mH P^2 \times S^1$ (if $n=9$)
and $M$ is spin bordant to $P = \emptyset$ or to 
$P = \mH P^2 \times S^1 \times S^1$ (if $n=10$).

Now let $W$ be a spin bordism from $P$ to $M$. By performing surgeries 
of dimension $0$, $1$, $2$, and $3$ one can find a spin bordism $W'$ from 
$P$ to $M$ which is $3$-connected, that is $W'$ is connected 
and $\pi_1(W')=\pi_2(W')=\pi_3(W')=0$. The inclusion $i:M\to W$ is thus 
$3$-connected, that is bijective on $\pi_i$ for $i\leq 2$ and surjective 
on $\pi_3$. This implies that $W'$ can be decomposed into handles each of 
which corresponds to a surgery of dimension $\leq n-4$.
\end{proof}

The following corollary extends similar results from 
\cite{ammann.dahl.humbert:p11b} which hold in dimension $n=7$, $n=8$ and 
$n\geq 11$. We define
$s_1 \definedas \si(\mH P^2 \times S^1)$ and
$s_2 \definedas \si(\mH P^2 \times S^1 \times S^1)$.

\begin{corollary}
Let $M$ be a 2-connected compact spin manifold of dimension $n=9$ or 
$n=10$ with $\al(M)=0$. Then
\[
\si(M) \geq 
\begin{cases}
\min\{\La_{9,1},\La_{9,2},\La_{9,3},\La_{9,4},\La_{9,5},s_1\} > 109.2 
&\text{for } n=9 , \\
\min\{\La_{10,1},\La_{10,2},\La_{10,3},\La_{10,4},\La_{10,5},\La_{10,6},s_2\} 
\geq 97.3 
&\text{for } n=10 .
\end{cases}
\]
\end{corollary}

\begin{proof}
Lemma~\ref{lemma9.10} implies 
\[
\si(M) \geq \min\{\La_{9,1},\La_{9,2},\La_{9,3},\La_{9,4},\La_{9,5},s_1\}
\] 
if $n=9$ and 
\[
\si(M) \geq 
\min\{\La_{10,1},\La_{10,2},\La_{10,3},\La_{10,4},\La_{10,5},\La_{10,6},s_2\}
\]
if $n=10$. The relations $\La_{9,1} > 109.2$ and $\La_{10,1} > 102.6$ 
follow from Examples~\ref{ex.num.2} and~\ref{ex.num.3}. The relations 
\[
\min\{\La_{9,2},\La_{9,3},\La_{9,4},\La_{9,5}\} > 109.4 > 109.2
\]
and 
\[
\min\{\La_{10,2},\La_{10,3},\La_{10,4},\La_{10,5},\La_{10,6}\} > 126.4 > 102.6
\]
follow from the product formula, see 
\cite[Corollary~3.3]{ammann.dahl.humbert:13}. From 
\cite[Theorem~1.1]{akutagawa.florit.petean:07} it follows that 
$s_k \geq \mu(\mH P^2\times \mR^k)$. To estimate $s_1$ for $n=9$ we apply 
results of \cite{petean:09a}. The quantities $V$ and $V_8$ in that paper 
satisfy
\[
(\frac{V}{V_8})^{2/9} = 0.9370...,
\]
see Appendix~\ref{app.hpn}. Thus, \cite[Theorem~1.2]{petean:09a}
tells us that
\[
s_1 \geq \mu(\mH P^2 \times \mR) 
\geq 0.9370 \mu(\mS^9) = 138.57... >109.2.
\]
An estimate for $s_2$ when $n=10$ is provided by 
\cite[Example after Theorem~1.7]{petean.ruiz:13}, namely 
\[
s_2 \geq \mu(\mH P^2 \times \mR^2) 
\geq 0.59 \mu(\mS^{10}) > 97.3 < \La_{10,1}.
\]
\end{proof}

In the case that $\al(M) \neq 0$ for $2$-connected $M$ it was shown in 
\cite[Theorem 1]{petean:03} that $\si(M)=0$.

In dimensions $n\leq 6$, $n\neq 4$, there are only a few $2$-connected 
compact manifolds, namely $S^3$, $S^5$, $S^6$, and connected sums of 
$S^3\times S^3$, all with their standard smooth structures. The conformal 
Yamabe constant for the product metric on $S^3\times S^3$,
\[
\mu(S^3\times S^3, \rho^3+\rho^3)=12 (2\pi^2)^{2/3} = 87.64646..., 
\]
follows from Obata's theorem \cite[Proposition~6.2]{obata:71.72}.
Using Theorem C or more precisely the third conclusion 
in the following unnumbered 
corollary of \cite{boehm.wang.ziller:04} 
we find
\[
\si(S^3\times S^3) > 12 (2\pi^2)^{2/3} = 87.64646...
\]

In all dimensions $\neq 4$ we thus obtain lower bounds for the smooth
Yamabe invariant. In dimensions $n=7$, $n=8$, and $n\geq 11$ an explicit 
lower bound for the smooth Yamabe invariant of $2$-connected compact 
manifolds with vanishing index was obtained in Corollaries~6.6, 6.7 
and Proposition~6.9 of \cite[Corollary~6.6]{ammann.dahl.humbert:p11b}.
Summarizing we have the following proposition.
\begin{proposition}\label{prop.2c.lowbound}
Let $M$ is a $2$-connected compact manifold of dimension $n\neq 4$. 
If $\alpha(M)\neq 0$, then $\si(M)=0$. If $\alpha(M)=0$, then 
$\si(M) \geq t_n$, where $t_n$ is an explicit positive number only 
depending on $n$.
\end{proposition}
Some values of $t_n$ are collected in Table~\ref{fig.tn}.
\begin{figure}
\begin{tabular}{|r||l|l|l|l|l|l|l|l|l|}
\hline
$n=$ & \phantom{8}3 &  \phantom{8}4 & \phantom{8}5 & \phantom{8}6 &  \phantom{8}7 &  \phantom{8}8  &  \phantom{8}9 &  \phantom{8}10 &  \phantom{8}11 \\
\hline
$\si(M)\geq t_n=$ & 43.8 & \phantom{8}? & 78.9 & 87.6 & \phantom{8}74.5 & \phantom{8}92.2 & 109.2 & \phantom{8}97.3 & 135.9 \\
\hline
$\si(S^n)=$ & 43.8 &  61.5 & 78.9 & 96.2 & 113.5 & 130.7 & 147.8 & 165.0 & 182.1\\
\hline
\end{tabular}
\caption{Lower estimates for the smooth Yamabe invariant of $2$-connected 
manifolds with vanishing index. Values of $\si(S^n)$, rounded down, for 
comparison}
\label{fig.tn} 
\end{figure}
 
The situation for $n=4$ is still unclear as it is unknown whether exotic 
$4$-spheres, i.e. manifolds homeomorphic but not diffeomorphic to $S^4$, 
do exist. The smooth Poincar\'e conjecture in dimension $4$ claims that 
exotic $4$-spheres do not exist. This would imply that $S^4$ is the only 
$2$-connected $4$-manifold and thus $t_4=\si(S^4)$.

%%%%%%%%%%%%%%%%%%%%%%%%%%%%%%%%%%%%%%%%%%%%%%%%%%%%%%%%%%%%%%%%%%%%%%%%%
\appendix
%%%%%%%%%%%%%%%%%%%%%%%%%%%%%%%%%%%%%%%%%%%%%%%%%%%%%%%%%%%%%%%%%%%%%%%%%

%%%%%%%%%%%%%%%%%%%%%%%%%%%%%%%%%%%%%%%%%%%%%%%%%%%%%%%%%%%%%%%%%%%%%%%%%
\section{Optimal values of $\la$ and $\tau$}
\label{app.opti}
%%%%%%%%%%%%%%%%%%%%%%%%%%%%%%%%%%%%%%%%%%%%%%%%%%%%%%%%%%%%%%%%%%%%%%%%%

We now optimize $\la$ and $\tau$ for the inequality 
\eqref{estimate-la-tau}. We define the convex polygon $P_\ci$ of 
admissible pairs $(\la,\tau)$ as
\[
P_\ci
\definedas
\{(\la,\tau) \mid \mbox{satisfying \eqref{cond1}, \eqref{cond2},
$\la\geq 0$, $\tau\geq 0$}\}.
\] 
For $\la=1$, $\tau=0$, one has $\la \ci^2 s_1 + \tau s_0 < s_\ci$ so 
$(1,0)$ is a corner of $P_\ci$. Similarly one sees that $(0,1)$ is never 
a corner of $P_\ci$ unless $\ci=0$. Because of $\ci^2 s_1/s_0<1$, the 
equations $\la+\tau=1$ and $\la \ci^2 s_1 + \tau s_0 = s_\ci$ have a 
common solution $(\la_0,\tau_0)$ with $\la_0\in (0,1)$ for 
$\ci\in (0,1)$. From
\[
\frac{\ci^{2w/n}\mu_1}{\mu_0}
\geq 
\ci^{2w/n}>\frac{\ci^2 s_1}{s_0}
\]
one easily sees that the optimal estimate is obtained in the point 
$(1,0)$ for $\ci^{2w/n}\geq \mu_0/\mu_1$, and in the point $(\la_0,\tau_0)$ 
for $\ci^{2w/n}\leq \mu_0/\mu_1$.

Next we compute $\la_0$.
\[
- \la_0 \ci^2 v(v-1) + \la_0 \ci^2 w(w-1) + (1-\la_0)w(w-1)
\leq 
- \ci^2 v (v-1)
\]
Factoring out, removing $w(w-1)$ on both sides, then dividing 
by $\la_0\ci^2 w(w-1)$ one obtains the equivalent equation
\[
-\frac{v(v-1)}{w(w-1)}+ 1 - \frac1{\ci^2} \leq - \frac{1}{\la_0} 
\frac{v(v-1)}{w(w-1)} ,
\]
which is further equivalent to 
\[
\left(1-\frac1{\ci^2}\right)
\leq
\left(1-\frac1{\la_0}\right)\frac{v(v-1)}{w(w-1)} .
\]
This yields \eqref{la0-eq}.

%%%%%%%%%%%%%%%%%%%%%%%%%%%%%%%%%%%%%%%%%%%%%%%%%%%%%%%%%%
\section{The constant $\La_{6,3}$} \label{nk=63}
%%%%%%%%%%%%%%%%%%%%%%%%%%%%%%%%%%%%%%%%%%%%%%%%%%%%%%%%%%%

All explicitly known positive lower bounds for $\La_{n,k}$  are obtained 
in the following way: at first, we show that 
$\La_{n,k}^{(2)} \geq \La_{n,k}^{(1)}$ and then, we apply 
Theorem \ref{theo.concl} or the estimates obtained in 
\cite{ammann.dahl.humbert:13}. Recall that by 
definition $\La_{n,k} = \min (\La^{(1)}_{n,k},\La^{(2)}_{n,k})$. 
For $0 \leq k \leq n-2$ or $(n,k) \in \{ (4,1),(5,2) \}$, the 
inequality  $\La_{n,k}^{(2)} \geq \La_{n,k}^{(1)}$ is a direct consequence 
of the main result in  \cite{ammann.dahl.humbert:p11b}. 
For $(n,k) = (6,3)$, this result does not apply directly, but a modified 
version which will be presented in this appendix
still allows to conclude $\La_{n,k}^{(2)} \geq \La_{n,k}^{(1)}$. 

\begin{prop} \label{prop_nk=63}
We have $\La^{(2)}_{6,3} \geq \La^{(1)}_{6,3}$ and hence $\La_{6,3} = \La^{(1)}_{6,3}$.
\end{prop}

\begin{proof} 
The main result in \cite{ammann.dahl.humbert:p11b} implies that 
\[ 
\inf_{c \in [0,1)} \mu^{(2)} (\mM_\ci ) \geq \La^{(1)}_{6,3}
\]
and as a consequence, $\La^{(2)}_{6,3} \geq \min(\La^{(1)}_{6,3}, \mu^{(2)}(\mM_1))$. 
We now estimate $\mu^{(2)}(\mM_1)$. 
If we spell out the definition of
 $\mu^{(2)}(\mM_1)$ recalled in Section~\ref{sec.notation}, and using
%$\scal^{G_1} = -10$, 
$a_6=5$, and $p_5=3$, we see that $\mu^{(2)}(\mM_1)$
is the infimum of all $\mu\in \mR$ for which there is a solution of 
\begin{equation} \label{eq.conf1} 
5 \Delta^{G_1} u + s^{G_1} u = \mu u^{2}
\end{equation}
satisfying
\begin{itemize}
\item $u \not \equiv 0$,
\item $\|u\|_{L^{3}(\mM_1)} \leq 1$,
\item $u \in L^{\infty}(\mM_1)$,
\item $\mu(u) \|u\|_{L^{\infty}( \mM_1)} \geq
\frac{5}{32}$. 
\end{itemize}

We prove in \cite{ammann.dahl.humbert:13b} that there is a conformal 
diffeomorphism $\Theta: \mS^6 \setminus \mS^3\to  \mM_1$ where $\mS^3$ 
denotes a totally geodesic $3$-sphere in the standard sphere $\mS^6$. 
Let $f\in C^\infty(\mS^6 \setminus \mS^3)$, $f>0$, be the conformal factor 
of $\Theta$, i.e.\ $\Theta^*G_1 = f \rho^6$. We define 
$v \definedas f \Theta^*u$. By conformal covariance of the Yamabe 
operator and since the scalar curvature of $\mS^6$ is $30$, we get 
from \eqref{eq.conf1} that the function $v$ is a solution of
\begin{equation} \label{eq.S6}
5 \Delta^{\rho^6} v +30 v = \mu v^{2}
\end{equation} 
on $\mS^6 \setminus \mS^3$. Moreover, one checks that 
\[
\|v\|_{L^{3}( \mS^6 \setminus \mS^3)}  = \|u\|_{L^{3}(\mM_1)}
\]
and hence, $v \in L^3(\mS^6)$ and 
\begin{equation} \label{vl6} 
\|v\|_{L^{3}( \mS^6)} \leq 1. 
\end{equation}
We now use a standard argument to show that the function $v$ can be 
extended to a smooth solution of equation \eqref{eq.S6} on all $S^6$. 
In other words, we remove the singularity at $\mS^3$. Let us choose a 
smooth function $\phi$ on $\mS^6$. We are going to show that 

\begin{equation} \label{vlphi}
\int_{S^6} v (L \phi)-\mu v^2 \phi\, dv =0
\end{equation}
where, to simplify notations, we set $L \definedas 5 \Delta^{\rho_6} + 30$ 
and where $dv \definedas dv^{\rho^6}$.
 
For all $a \geq 0$, let us denote by $W_a$ the set of points of $S^6$ 
whose distance to the removed $\mS^3$ is smaller than $a$. For this goal, 
consider for $\ep\in(0,\frac12)$  a cut-off function $\eta_\ep$ such that 
\begin{enumerate}
\item $0 \leq \eta_\ep\leq 1$;
\item $\eta_{\ep} (\mS^6 \setminus W_{2 \ep}) = \{0\}$;
\item $\eta_\ep(W_\ep) = \{ 1 \}$; 
\item $|\nabla \eta_\ep | \leq 2/ \ep$.
\item $|\nabla^2 \eta_\ep | \leq c/ \ep^2$.
\end{enumerate}

We then write, for $\ep>0$ small

\begin{equation} \label{vlphi2}
\int_{S^6} v( L \phi)\,dv 
= \int_{S^6} v L  (\eta_\ep \phi + (1-\eta_\ep) \phi) \,dv.
\end{equation} 

Since $v$ satisfies Equation \ref{eq.S6} and since the function 
$1-\eta_\ep$ is compactly supported in $\mS^6 \setminus \mS^3$,  we have 
\[\begin{split} 
\int_{S^6} v L  ((1-\eta_\ep) \phi) \,dv 
&= 
\int_{\mS^6} (Lv) (1-\eta_\ep) \phi) \,dv\\
&=  
\int_{\mS^6} \mu v^2  (1-\eta_\ep) \phi) \,dv. 
\end{split}\]

Since $1-\eta_{\ep}$ is bounded and tends to $1$ almost everywhere, 
Lebesgue's theorem implies 

\begin{eqnarray} \label{vlphi3} 
 \lim_{\ep \to 0} v L ((1-\eta_\ep) \phi) \,dv = \int_{\mS^6} \mu v^2 \phi \,dv.
\end{eqnarray} 

Now, we use the fact that there exists some $C>0$ independent of $\ep$, 
but depending on $\phi$, such that 
\[
L(\eta_\ep \phi) \leq C (\frac{\chi_\ep }{\ep^2} + \eta_\ep) 
\]
where $\chi_\ep$ is the characteristic function of the set 
$W_{2 \ep} \setminus W_\ep$. 

Then, using H\"older inequality and the fact that $\eta_\ep$ is compactly 
supported in $W_{2\ep}$ and bounded by $1$ on this set,   
\[ \begin{split} 
\int_{S^6} v L  (\eta_\ep \phi)\, dv 
&\leq 
C\left( \frac{1}{\ep^2}  \int_{W_{2\ep}} v\, dv 
+ \int_{\mS^6} v \eta_{\ep} \,dv \right) \\
&\leq  
C \left( \frac{1}{\ep^2} \left(\int_{W_{2\ep}} v^3  \,dv\right)^{1/3} 
\vol(W_{2 \ep})^{2/3}  
+ \left(  \int_{W_{2\ep}} v^3  \,dv\right)^{1/3} \vol(W_{2 \ep})^{2/3}\right) \\
&\leq 
C  \frac{1}{\ep^2} \left( \int_{W_{2\ep} } v^3  \,dv\right)^{1/3} 
\vol(W_{2 \ep})^{2/3}.
\end{split} \]

Since $W_{2\ep}$ is a $2\ep$-neighborhood of $\mS^3$, 
$ \vol(W_{2 \ep}) \leq C \ep^3$. Moreover, since $v \in L^3(\mS^6)$, 
\[
\lim_{\ep \to 0}  \int_{W_{2\ep} } v^3  \,dv =0.
\]
We then obtain that %\eqref{vlphi}-
\[
\lim_{\ep \to 0} \int_{S^6} v L  (\eta_\ep \phi) \,dv=0.
\]
Together with \eqref{vlphi3} and \eqref{vlphi2}, we obtain \eqref{vlphi} 
which means that in the sense of distributions,
equation \eqref{eq.S6} is satisfied on all of $\mS^6$.
By standard elliptic theory, $v$ is $C^2$ (and even smooth 
outside its zero set). Using $v$ as a test function in the Yamabe 
function of $ \mS^6$, we get from  \eqref{eq.S6} and \eqref{vl6} that 
$\mu \geq \mu(\mS^6) \geq \La^{(1)}_{6,3}$, which ends the proof. 
\end{proof}

%%%%%%%%%%%%%%%%%%%%%%%%%%%%%%%%%%%%%%%%%%%%%%%%%%%%%%%%%%%%%%%%%%%%%%%%%
\section{The Wu manifold $\SU(3)/\SO(3)$}
\label{app.wu.mfd}
%%%%%%%%%%%%%%%%%%%%%%%%%%%%%%%%%%%%%%%%%%%%%%%%%%%%%%%%%%%%%%%%%%%%%%%%%

We equip $\SU(3)$ with the bi-invariant metric such that the matrix 
\[
\begin{pmatrix}
0 & -1 & 0\\
1 & 0 & 0\\ 
0 & 0 & 0
\end{pmatrix}
\in \mathfrak{su}(3)
\]
has length $1$. Then $(\SU(3),\SO(3))$ is a symmetric pair, and the 
associated involution of $\mathfrak{su}(3)$ is complex conjugation. 
Let $M$ be $\SU(3)/\SO(3)$ equipped with the quotient metric $g$. 
The manifold $M$ is orientable, but not spin. Complex conjugation of 
$\SU(3)$ induces an orientation reversing isometry of $M$. Thus 
$M \amalg M$ is (up to orientation-preserving diffeomorphisms) the 
oriented boundary of $M\times [0,1]$. It follows that $M\# M$ is an 
oriented boundary as well.

An elementary calculation on the Lie algebra level shows that $g$ 
is an Einstein metric, $\Ric^g = 6g$. Obata's theorem 
\cite[Proposition~6.2]{obata:71.72} then tells us that $\mu(M,g) = 
30 \vol(M,g)^{2/n}$. The volume $\vol(M,g)$ is calculated in 
\cite{boya.et.al:03}, and we conclude the following Lemma.

\begin{lemma}
The conformal Yamabe constant of $\SU(3)/\SO(3)$ is 
\[
\mu(\SU(3)/\SO(3),g) 
= 30 \cdot \left(\frac{\sqrt 3}8 \pi^3\right)^{2/5}
= 64.252401...
\]
\end{lemma}

%%%%%%%%%%%%%%%%%%%%%%%%%%%%%%%%%%%%%%%%%%%%%%%%%%%%%%%%%%%%%%%%%%%%%%%%%
\section{Quaternionic projective spaces $\mH P^n$}
\label{app.hpn}
%%%%%%%%%%%%%%%%%%%%%%%%%%%%%%%%%%%%%%%%%%%%%%%%%%%%%%%%%%%%%%%%%%%%%%%%%

Let $g_n$ be the metric on $\mH P^n$ such that the Hopf map 
$\mS^{4n+3}\to \mH P^n$ is a Riemannian submersion. With O'Neill's 
formula one easily calculates that the scalar curvature of $g_n$ is 
$\scal^{g_n}=4n(4n+8)$, and the volume is 
$\vol(\mH P^n,g_n) = \omega_{4n+3}/\omega_3$.

As a consequence
\[
\ti g_n \definedas
\frac{\scal^{g_n}}{\scal^{\rho^{4n}}} g_n = \frac{4n(4n+8)}{4n (4n-1)}g_n
\]
is a metric whose scalar curvature is equal to $4n(4n-1)=\scal^{\rho^{4n}}$.
Its volume is 
\[
V_{4n} \definedas 
\vol(\mH P^n,\ti g_n) = 
\left(\frac{4n+8}{4n-1}\right)^{2n}\frac{\omega_{4n+3}}{\omega_3}.
\]
In the special case $n=2$ this yields 
$V_8 = 2^{13}\pi^4/(7^4 \cdot 5 \cdot 3)$ where we used 
$\om_{11} = \pi^6/60$ and $\om_3 = 2\pi^2$. 
Using $\om_8 = 32\pi^4/(7 \cdot 5 \cdot 3)$ we obtain 
$V_8/\om_8 = 2^8/7^3 = 0.74635569\ldots$. These numbers play a crucial 
role for the lower bounds of $\mu(\mH P^2\times \mR)$ and 
$\mu(\mH P^2\times \mR^2)$.

%%%%%%%%%%%%%%%%%%%%%%%%%%%%%%%%%%%%%%%%%%%%%%%%%%%%%%%%%%%%%%%%%%%%%%%%%%
%\bibliographystyle{amsplain}
%\bibliography{literatur}
%%%%%%%%%%%%%%%%%%%%%%%%%%%%%%%%%%%%%%%%%%%%%%%%%%%%%%%%%%%%%%%%%%%%%%%%%%

\providecommand{\bysame}{\leavevmode\hbox to3em{\hrulefill}\thinspace}
\providecommand{\MR}{\relax\ifhmode\unskip\space\fi MR }
\providecommand{\MRhref}[2]{%
  \href{http://www.ams.org/mathscinet-getitem?mr=#1}{#2}
}
\providecommand{\href}[2]{#2}

%%%%%%%%%%%%%%%%%%%%%%%%%%%%%%%%%%%%%%%%%%%%%%%%%%%%%%%%%%%%%%%%%%%%%%%%%
\end{document}